\documentclass[12pt,twoside]{article} 
\usepackage{amsthm,geometry}
\usepackage{mathtools}
\newtheorem{Theorem}{Theorem}[section]
\newtheorem{Lemma}{Lemma}[section]
\newtheorem{Corollary}{Corollary}[Lemma]

\date{} 

\begin{document} 

\centerline{\bf International Journal of Contemporary Mathematical Sciences}

\centerline{\bf Vol. 15, 2020, no. 3, 113 - 120}

\centerline{\bf HIKARI Ltd, \ www.m-hikari.com}

\centerline{\bf https://doi.org/10.12988/ijcms.2020.91439}

\centerline{} 

\centerline{} 

\centerline{\Large{\bf The Gap $g_n$ Between Two Consecutive Primes}}

\centerline{}

\centerline{\Large{\bf Satisfies $g_n=\mathcal{O}({p_n}^{2/3})$}} 

\centerline{} 

\centerline{\bf {Madieyna Diouf}} 

\centerline{} 

\centerline{School of Mathematical and Statistical Sciences } 

\centerline{Arizona State University} 

\centerline{P.O. Box 871804} 

\centerline{Tempe, AZ 85287-1804, USA} 

\centerline{} 
{\footnotesize This article is distributed under the Creative Commons by-nc-nd Attribution License. Copyright $\copyright$ 2020 Hikari Ltd.}

\begin{abstract} 
The following is proven using arguments that do not revolve around the Riemann Hypothesis or Sieve Theory. If $p_n$ is the $n^{\rm th}$ prime and $g_n=p_{n+1}-p_n$, then  $g_n=\mathcal{O}({p_n}^{2/3})$. 
\end{abstract} 



\section{Introduction} 
The study of maximal gaps between consecutive primes is an important subject that is actively pursued and the Bertrand's postulate \cite{Bertrand} is one of its first consequences. In $1850$, Chebyshev proved the Bertrand’s postulate \cite{Tchebychev}, and  P. Erd\" os presented a simplified proof in $1932$  \cite{Erdos}. Strong results were also obtained in the generalizations of Bertrand's Postulate. In  $2006$, M. El Bachraoui proved the existence of a prime in the interval $[2n, 3n]$ \cite{Bachraoui}.  In $2011$, Andy Loo exhibited a proof that shows not only the existence of a prime between $3n$ and $4n$, but also the infinitude of the number of primes in this interval when $n$ goes to infinity  \cite{Andy}. Pierre Dusart gave the best known result in this category when he improved in $2016$ his previous work by showing that there is a prime between $x$ and $(x+x/(25{\log}^2 x))$ for $x\geq468991632$ \cite{Pierre}.\\

\lq\lq On $25th$ October $1920$ G. H. Hardy read Cram\'er's paper \textit{\lq\lq On the distribution of primes"} to the Cambridge Philosophical Society. Here Cram\'er develops a statistical approach to this question showing that for any $\epsilon>0$
\begin{center}
$ p_{n+1}-p_n=O(p^{\epsilon}$)
\end{center}
for \lq most\rq \: $p_n$: in fact for all but at most $x^{1-3\epsilon/2}$ of the primes $p_n\leq x$.\rq\rq \cite{Andrew}. As a result of the Prime Number Theorem alone, we have
$p_{n+1}-p_n<\epsilon p_n$ for all $\epsilon>0$.\\
By the  Prime Number Theorem with error term, we obtain 
\begin{center}
$p_{n+1}-p_n<\frac{p_n}{{(\log p_n)^c}}$ \textup{\:  for some positive constant $c$}.
\end{center}

Significant works have been done on the upper bound of the gap between consecutive primes by various authors without assuming an unproved hypothesis. Hoheisel was the first to show in $1930$ the existence of a constant $\delta>0$ (mainly $\delta = 1/33000$) such that $p_{n+1}-p_n = O({p_n}^{1-\delta})$ \cite{Hoheisel}. Heilbronn \cite{Heibronn}, and Tchudakoff \cite{Tchuda}, both improved on the value of $\delta$. Ingham \cite{Ingham} made a significant progress that contributed to the first solutions surrounding the problem of existence of a prime between two consecutive cubes.\\

Based on observations uniquely centered on the midpoint $m$ of two consecutive primes $p_n$, $p_{n+1}$ and the largest odd multiple of $p_n$ not exceeding $m^2$, we exercise basic proof techniques to show that the gap $g_n$ between two consecutive primes satisfies $g_n=O({p_n}^{2/3})$. It is indeed shown precisely that 
\begin{center}
${g_n}^3<16{p_n}^2$.
\end{center}

Better results are obtained under the assumption of the Riemann Hypothesis. Harald Cram\'er proved that if the Riemann hypothesis holds, then the gap $g_n$ satisfies $g_n=O(\sqrt{p_n}\log p_n)$.  Results of the form $g_n=\mathcal{O}({p_n}^{\theta})$ with different $\theta <1$ were given in the past. Among these where $\theta$ is close to $1/2$ are namely from \cite{Hu} $\theta=7/12$, Dr Brown gave an alternative proof, \cite{He} $\theta=7/12$, \cite{Mo} $\theta=1051/1920$, \cite{Qi} $\theta=6/11$,   The best unconditional bound is known to Baker, Harman and Pintz, who proved the existence of $x_0$ such that there is a prime in the interval $[x, x+O(x^{21/40})]$ for $x>x_0$, \cite{Baker}.\\ 

The key ideas that allow us to give an explicit and  unconditional result using methods that can be easy exercises in first course in elementary number theory can be summarized as following. The use of the midpoint $m$ of the two consecutive primes $p_n$ and $p_{n+1}$, combined with few other elementary manipulations of the largest odd multiple of $p_n$ not exceeding $m^2$.\\
Lemma $4$ establishes the relation 
\begin{align}
g_n&<2\sqrt{2p_{n}(c_n+1)}.
\end{align}
Proceedings from Lemma $1$ to Lemma $3$ lead to a modest upper bound of $c_n$. 
\begin{align}
c_n&<p_n/g_n.
\end{align}
Theorem $1$ is a combination of $(1)$ and $(2)$, that gives the desired result and implies the existence of a prime between two consecutive cubes for all positive integer $n>16$. 

\section{Main results}
Consider two consecutive primes $p_n$ and $p_{n+1}$, set $m$ to be their point fixed, once for all. \\
There exists a positive integer $b$ such that $m-b=p_{n}$ and $m+b=p_{n+1}$.
\begin{align}
(m-b)(m+b)&=p_{n}p_{n+1}.\\
m^2-b^2&=p_{n}p_{n+1}.\\
m^2-p_{n}p_{n+1}&=b^2.\\
m^2-p_{n}p_{n+1}&>0.\\
p_{n}p_{n+1}&<m^2.
\end{align}
\textbullet \: Let $\alpha p_n$ denote the largest odd multiple of $p_n$ not exceeding $m^2$.\\
\textbullet \: Let $\beta p_{n+1}$ denote the largest odd multiple of $p_{n+1}$ not exceeding $m^2$.\\
\textbullet \: Set $c_n$ to be the number of odd multiples of $p_n$ between $p_np_{n+1}$ and $m^2$.\\
\textbullet \: Set $c_{n+1}$ to be the number of odd multiples of $p_{n+1}$ between $p_np_{n+1}$ and $m^2$.\\\\
Combining $(7)$ with the previous four sentences shows that $\alpha p_n$ and $\beta p_{n+1}$ are between $p_np_{n+1}$ and $m^2$, and 
\begin{align}
\alpha p_n&=p_n(p_{n+1}+2c_n).\\
\beta p_{n+1}&=p_{n+1}(p_n+2c_{n+1}).\\
\beta p_{n+1}-\alpha p_n&=2(p_{n+1}c_{n+1}-p_nc_n).
\end{align}
\begin{Lemma}
 $c_{n+1}\leq c_n$.
\end{Lemma}
\begin{proof}
In the contrary, suppose that $c_{n+1}>c_n$.
\begin{align}
\alpha p_n&=p_n(p_{n+1}+2c_n) \text{ \: \: \: \:  (by $(8))$}.\\
\beta p_{n+1}&=p_{n+1}(p_n+2c_{n+1}) \text{ \: \: (by $(9))$}.\\
&\geq p_{n+1}(p_n+2(c_n+1)).\\
&\geq p_{n+1}(p_n+2c_n)+2p_{n+1}.\\
&\geq \alpha p_n+2p_{n+1}.\\
&\geq \alpha p_n+2p_n.\\
\beta p_{n+1}&> m^2 \textup{\: (by definition, $\alpha p_n+2p_n>m^2$.)}
\end{align}
Observe that $(17)$ is in contradiction with the definition of $\beta p_{n+1}$, that is the largest odd multiple of $p_{n+1}$ not exceeding $m^2$; Therefore,  $c_{n+1}\leq c_n$.\\
\end{proof}
\centerline{\textbullet \: Set $X_n=(m^2 - p_n) mod(2p_n)$}
\centerline{ \: \: \: \: \: \: \textbullet \: and $X_{n+1}=(m^2 - p_{n+1}) mod(2p_{n+1}).$ }
\begin{Lemma}
$\beta p_{n+1}-\alpha p_n=X_n-X_{n+1}.$
\end{Lemma}
\begin{proof}
The value $X_n$ gives the distance between $m^2$ and the largest odd multiple of $p_n$ not exceeding $m^2$. Hence,
$m^2-X_n$ is the largest odd multiple of $p_n$ not exceeding $m^2$; That is $\alpha p_n$ by definition. 
\begin{align}
\alpha p_n&= m^2 - X_n.
\end{align}
Similarly, $X_{n+1}$ represents the distance between $m^2$ and the largest odd multiple of $p_{n+1}$ not exceeding $m^2$. Thus, 
$m^2-X_{n+1}$ is the largest odd multiple of $p_{n+1}$ not exceeding $m^2$; That is $\beta p_{n+1}$ by definition. 
\begin{align}
\beta p_{n+1}&= m^2 - X_{n+1}.
\end{align}
$(18)$ and $(19)$ give
\begin{align}
\beta p_{n+1}-\alpha p_n&=X_n-X_{n+1}.
\end{align}
\end{proof}
\begin{Corollary}   
 $\beta p_{n+1}-\alpha p_n<2p_n.$
\end{Corollary}
\begin{proof}
By Lemma $2.2$
\begin{align}
\beta p_{n+1}-\alpha p_n&=X_n-X_{n+1}.\\
\text{Since\: \: } X_n= (m^2 - p_n) mod(2p_n)&<2p_n\\
\text{and\: \: }  X_{n+1}=(m^2 - p_{n+1}) mod(2p_{n+1})&>0,\\
\text{then\: \: } \beta p_{n+1}-\alpha p_n&<2p_n. 
\end{align}
\end{proof}
\begin{Corollary}   
 $\beta p_{n+1}-\alpha p_n=2c_ng_n.$
\end{Corollary}
\begin{proof}
It is proven unconditionally in Lemma $2.2$  that 
\begin{align*}
\beta p_{n+1}-\alpha p_n&=X_n-X_{n+1}.
\end{align*}
 The result above  is obtained without reference to $c_{n+1}=c_n$ or $c_{n+1}<c_n$; Therefore,
\begin{align}
\text{if $c_{n+1}=c_n$, then \: } \beta p_{n+1}-\alpha p_n &=X_n-X_{n+1};\\
\text{Otherwise (that is if $c_{n+1}<c_n$), \: } \beta p_{n+1}-\alpha p_n &=X_n-X_{n+1}.
\end{align}
Moreover, 
\begin{align}
X_n-X_{n+1} &=(m^2 - p_n) mod(2p_n)-(m^2 - p_{n+1}) mod(2p_{n+1}).\\
X_n-X_{n+1} &=((\frac{p_n+p_{n+1}}{2})^2 - p_n) mod(2p_n)-((\frac{p_n+p_{n+1}}{2})^2 - p_{n+1}) mod(2p_{n+1}).
\end{align}
It is clear from $(28)$ that $X_n-X_{n+1}$ is independent of $c_n$ and $c_{n+1}$ and it is solely in terms of the given consecutive primes $p_n$ and $p_{n+1}$. \\
Therefore, the value $X_n-X_{n+1}$ does not change whether $c_{n+1}=c_n$ or $c_{n+1}<c_n$.\\
The previous sentence with $(25)$ and $(26)$ imply that \\
\centerline{a)The quantity $\beta p_{n+1}-\alpha p_n$ if $c_{n+1}=c_n$, is the same as it would be if $c_{n+1}<c_n$.}
\centerline{Statement a) above combined with Lemma $2.1$, imply the following.}
\centerline{b)The quantity $\beta p_{n+1}-\alpha p_n$ can be obtaining by letting $c_{n+1}=c_n$. \: \: \: \: \: \: \: \: }
\begin{align}
\text{Thus, from \:  } \beta p_{n+1}-\alpha p_n&=2(p_{n+1}c_{n+1}-p_nc_n) \text{\: \: (by $(10)$)}.\\
\text{We obtain\: \: } \beta p_{n+1}-\alpha p_n&=2(p_{n+1}c_n-p_nc_n) \text{\: \: \: \: (Stat b) $c_{n+1}=c_n)$}.\\
\text{So that\: \: } \beta p_{n+1}-\alpha p_n &=2c_ng_n. 
\end{align}
\end{proof}
\begin{Lemma}
 $c_n<p_n/g_n.$
\end{Lemma}
\begin{proof}
\begin{align}
 \beta p_{n+1}-\alpha p_n&<2p_n \text{\: \: \:(Corollary $2.2.1$).}\\
 \beta p_{n+1}-\alpha p_n&=2c_ng_n \text{\: (Corollary $2.2.2$).}\\
\text{$(32)$ and $(33)$ imply that\: \:  \: \: \: } 2c_ng_n&<2p_n.\\
c_n&<p_n/g_n.
\end{align}
\end{proof}
\begin{Lemma}
$g_n<2\sqrt{2p_{n}(c_n+1)}.$
\end{Lemma}
\begin{proof}
We have $\alpha p_n=p_n(p_{n+1}+2c_n)$ is the largest odd multiple of $p_n$ not exceeding $m^2$.  
\begin{align}
p_{n}(p_{n+1}+2c_n)+2p_{n} &> m^2.\\ 
m^2 &<  p_{n}(p_{n+1}+2c_n+2) .\\
m^2-p_np_{n+1} &<  2p_{n}(c_n+1) .\\
m^2-p_n(2m-p_n) &<  2p_{n}(c_n+1) .\\
m^2-2mp_n+{p_n}^2 &<  2p_{n}(c_n+1) .\\
(m-p_n)^2 &<  2p_{n}(c_n+1) .\\
(g_n/2)^2 &<  2p_{n}(c_n+1) .\\
(g_n/2)^2 &<  2p_{n}(c_n+1) .\\
g_n&<2\sqrt{2p_{n}(c_n+1)}.
\end{align}
\end{proof}
\begin{Theorem}
$g_n=\mathcal{O}({p_n}^{2/3})$.
\end{Theorem}
\begin{proof}
\begin{align}
\textup{By Lemma $2.3$,\: \: \: \: \: \: }  c_n&<{p_n}/{g_n}; \text{\: That is \: } c_n+1<{2p_n}/{g_n}.\\
\textup{By Lemma $2.4$,\: \: \: \: \: \: }   g_n&<2\sqrt{2p_n(c_n+1)}.\\
\textup{With $(45)$ and $(46 )$,\:  \: \: \:  \: \: }  g_n&<2\sqrt{2p_n\frac{2p_n}{g_n}}.\\
{g_n}^3&<16{p_n}^2.\\
g_n&\ll {p_n}^{2/3}.
\end{align}
\end{proof}

\section{Primes between two consecutive cubes for all positive integers}
Let $K=16^{1/3}$ and let $N>K^3$ (that is $N>16$) be a positive integer. \\Let $p_n$  be the largest prime less than $N^3$. Then 
\begin{align}
p_n<N^3 < p_{n+1}&. \\
p_{n+1} &< K p_n^{2/3} + p_n \text{\: (by $(48)$)}.\\
\textup{Since \: } p_n&<N^3 \text{\: \: \: \: \: \:  \: \:(by $(50)$)},\\
\textup{then \: } K{p_n}^{2/3}&<KN^{2}.\\
\textup{$(51)$, $(52)$ and $(53)$ yield\: \: \: \:  } p_{n+1} &< N^3+KN^{2}.\\
p_{n+1}&< N^3 + 16^{1/3}N^2.\\
p_{n+1}&< N^3 +3N^2 +3N.
\end{align}
\begin{align}
N^3<p_{n+1}&<(N+1)^3-1 \text{\: \: \:  for all $N>16$.}
\end{align}
The presence of $(57)$ and a manual verification for $N\leq 16$ complete the argument.  


{\bf Received: April 21, 2020; Published: May 21, 2020}


\begin{thebibliography}{99} 
\bibitem{Baker}
Baker, R. C., Harman, G., Pintz, J., {The difference between consecutive primes, II}, {\it Proceedings of the London Mathematical Society}, {\bf 83} (3) (2001), 532-562.
https://doi.org/10.1112/plms/83.3.532 

\bibitem{Pierre}
Dusart, Pierre, {Explicit estimates of some functions over primes}, {\it The Ramanujan Journal}, (2016). https://doi.org/10.1007/s11139-016-9839-4.

\bibitem{Bachraoui}
El Bachraoui, M., {Primes in the interval [2n, 3n]}, {\it Int. J. Contemp. Math. Sci.}, {\bf 1} (13-16) (2006), 617-621. \\ https://doi.org/10.12988/ijcms.2006.06065 

\bibitem{Andrew}
Granville, Andrew, {Harald Cram{\'e}r and the distribution of prime numbers}, {\it Scandinavian Actuarial Journal}, {\bf 1} (1995),  12-28. \\
https://doi.org/10.1080/03461238.1995.10413946 

\bibitem{He}
Heath-Brown, D. R., {The number of primes in a short interval}, {\it J. Reine Angew. Math.}, {\bf 389} (1988), 22-63. \\
https://doi.org/10.1515/crll.1988.389.22 

\bibitem{Heibronn}
Heilbronn, H. A., \emph{Uber den Primzahlsatz von Herrn Hoheisel}, {\it Mathematische Zeitschrift}, {\bf 36} (1) (1933), 394-423. \\ https://doi.org/10.1007/BF01188631

\bibitem{Hoheisel}
Hoheisel, G., {Primzahlprobleme in der Analysis}, {\it Sitzunsberichte der Koniglich Preussischen Akademie der Wissenschaften zu Berlin}, {\bf 33} (1930), 3-11.

\bibitem{Hu}
Huxley, Martin N., {On the difference between consecutive primes,} {\it Inventiones mathematicae}, {\bf 15} (2) (1971), 164-170. \\
https://doi.org/10.1007/bf01418933 

\bibitem{Ingham}
Ingham, A. E., {On the difference between consecutive primes,} {\it Quarterly Journal of Mathematics}, Oxford Series, {\bf 8} (1) (1937), 255-266. https://doi.org/10.1093/qmath/os-8.1.255

\bibitem{Bertrand}
Joseph Bertrand, {Mémoire sur le nombre de valeurs que peut prendre une fonction quand on y permute les lettres qu'elle renferme,} {\it Journal de l'Ecole Royale Polytechnique}, Cahier 30, {\bf 18} (1845), 123-140.

\bibitem{Andy}
Loo, Andy, \emph{On the Primes in the Interval (3n, 4n)}, {\it International Journal of Contemporary Mathematical Sciences}, {\bf 6} (38) (2011), 1871-1882.

\bibitem{Mo}
Mozzochi, C. J., {On the difference between consecutive primes}, {\it Journal of Number Theory} {\bf 24} (2) (1986), 181-187.
https://doi.org/10.1016/0022-314x(86)90101-0 

\bibitem{Erdos}
 P. Erdos, {Beweis eines satzes von Tschebyschef}, {\it Acta Litt. Univ. Sci., Szeged, Sect. Math.}, {\bf 5} (1932), 194-198.

\bibitem{Tchebychev}
P. Tchebychev, {Mémoire sur les nombres premiers}, {\it Journal de mathématiques pures et appliquées}, Sér. 1 (1852), 366-390 (proof of the postulate: 371-382). 

\bibitem{Qi}
Qi, Yao, and Lou Shituo, {A Chebychev's type of prime number theorem in a short interval II,} {\it Hardy-Ramanujan Journal}, {\bf 15} (1992).

\bibitem{Tchuda}
Tchudakoff, N. G., {On the difference between two neighboring prime numbers}, {\it Math. Sb.}, {\bf 1} (1936), 799-814.





\end{thebibliography}
\end{document}